\documentclass[12pt]{amsart}

\usepackage{amssymb,latexsym}
\usepackage{enumerate}
\usepackage{amsmath}
\usepackage{amsthm}
\usepackage{amsfonts}
\usepackage{amsxtra}
\usepackage[mathscr]{eucal}

\newtheorem{thm}{Theorem}[section]
\newtheorem{lem}[thm]{Lemma}
\newtheorem{cor}[thm]{Corollary}

\newtheorem{defi}[thm]{Definition}

\newtheorem{exa}[thm]{Example}
\newtheorem*{rem}{Remark}

\begin{document}

\title{Power Values of certain Quadratic Polynomials}
\author{Anthony Flatters}
\address{School of Mathematics, University of East Anglia, Norwich NR4 7TJ, UK}
\email{Anthony.Flatters@uea.ac.uk}
\thanks{I would like to thank Professors Walsh, Gy\"{o}ry, Hajdu and Pint\'{e}r for sending me comments on an original draft of this paper and for also providing me with additional references.}

\begin{abstract}
In this article we compute the $q$th power values of the quadratic polynomials $f$ with negative squarefree discriminant such that $q$ is coprime to the class number of the splitting field of $f$ over $\mathbb{Q}$.
The theory of unique factorisation and that of primitive divisors
of integer sequences is used to deduce a bound on the values of $q$ which is small enough to allow the remaining cases to be easily checked.
The results are used to determine all perfect power terms of certain polynomially generated integer sequences, including the Sylvester sequence.

\end{abstract}



\maketitle

\section{Introduction}
In 1926, Siegel~\cite{52.0149.02} proved that an affine curve of genus at least one has only finitely many integer points. Siegel's theorem is ineffective; it gives us no way of explicitly determining all the integer points on such a curve. The equation \begin{equation} \label{eq:maineqn}
y^q=f(x)\textrm{,} \end{equation} where $q\geqslant 3$ and $f(x)$ is a quadratic polynomial with two distinct roots defines an affine curve with genus
\begin{displaymath}
\frac{(q-1)(q-2)}{2}\geqslant 1 \textrm{,} \end{displaymath}
hence (\ref{eq:maineqn}) has only finitely many integer solutions. The exact determination of all integer solutions to such an equation is generally very difficult. In \cite{MR0234912}, Baker gave the first explicit upper bounds for the integer solutions of equation (\ref{eq:maineqn}) in the case where $f$ has at least two simple zeros. This result was obtained using his theorem about lower bounds for linear forms in the logarithms of algebraic numbers. Given the nature of the theory, these bounds are typically very large and the following upper bound was derived
\begin{displaymath}
\max\{\vert x \vert , \vert y \vert \}< \exp\exp((5q)^{10}(n^{10n}H)^{n^2}) \textrm{,} \end{displaymath}
where $n=\deg(f)$, $H$ is the height of $f$. Since Baker's result, there have been many refinements to the theory of linear forms in logarithms which allow smaller upper bounds to be obtained, for example, see the papers \cite{MR759041,MR1458749,MR1116106,MR0480426,MR1348763}. In \cite{MR1631771}, a method is given for the complete determination of integral solutions to an equation of the form $ay^p=f(x)$ where $a\in\mathbb{Z}\setminus\{0\},p\geqslant 3$ and $f(x)$ is separable of degree at least 2. We will approach the problem of finding solutions differently. It was first proved by Tijdeman~\cite{MR0560494}, using Baker's transcendence methods, that if $f$ has at least 2 simple rational zeros and if $y^q=f(x)$ has an integer solution with $\vert y\vert>1$, then $q$ is bounded above by a computable constant depending only on $f$. Later this was improved by Schinzel and Tijdeman~\cite{MR0422150}, who showed that for $P(x)\in\mathbb{Q}[x]$ with at least 2 distinct zeros, an integer solution $\vert y\vert>1$, to the equation $y^m=P(x)$ implies $m$ is bounded by an effectively computable constant depending only on $P$. However, their technique uses lower bounds for linear forms in logarithms and the bound for $m$, is once again, very large. There have been several improvements to Schinzel and Tijdeman's result. In \cite{MR1657483}, the authors prove that if $f$ is a monic irreducible polynomial of degree $n\geqslant 2$, $b\in\mathbb{Z}\setminus\{ 0\}$ and $y^z=bf(x)$ in integers $x,y,z$, $z>1$, then $z<cM^{3n}(\log\vert 2b \vert )^3$ where $c$ is an effectively computable constant depending only on $n$ and the Mahler measure $M$ of $f$. See also \cite{MR1119684} for the result that if $f(x)$ is a monic irreducible polynomial of degree $n\geqslant 2$ then $y^z=f(x)$ in integers, then $z<(6n^3)^{30n^3}\vert D(f) \vert^{5n^2}$, where $D(f)$ is the discriminant of $f$. For further results on this area consult \cite{MR2264658,MR2107952,MR2310662,DEqns,MR891406}. In this paper, we use Bilu, Hanrot and Voutier's wonderful theorem about prime appearance in Lucas sequences (see \cite{MR1863855}) to give a very small bound on the exponent $q$ in equation (\ref{eq:maineqn}) (independent of the equation) in the case where $f$ is a quadratic polynomial whose discriminant belongs to a subset of the negative integers.
The bound on $q$ is small enough to allow a bare-hands approach to computing all integral solutions, and in particular does not use any transcendence methods directly. We begin by stating our most general result.
\begin{thm} \label{thm:mainthm1}
Let $f$ be a monic quadratic polynomial with integral coefficients such that $D(f)$ is negative and squarefree with the property that the class number $h$ of $\mathbb{Q}(\sqrt{D(f)})$ is greater than one. Let $q\geqslant 2$ be a prime, and assume $x,y$ are integers such that
\begin{displaymath}
y^q=f(x)\textrm{.} \end{displaymath}
Then $q\leqslant \max\{3,P(h)\}$, where $P(h)$ denotes the greatest prime factor of $h$. \end{thm}
We now consider the case where the ring of integers of the splitting field of $f$ is a unique factorisation domain.
\begin{thm}\label{thm:mainthm}
Let $f$ be a monic quadratic polynomial with integer coefficients. Further, suppose that $-D(f)\in\{7,11,19,43,67,163\}$. If the equation (\ref{eq:maineqn}) is soluble in integers $x,\vert y\vert >1,q>2$ prime, then $q\leqslant q_0$, where
\begin{displaymath}
q_0=\left\{ \begin{array}{ll}
13 & \textrm{if $D(f)=-7$,} \\
7 & \textrm{if $D(f)=-19$,} \\
5 & \textrm{if $D(f)=-11$,} \\
4 & \textrm{if $D(f)=-43,-67,-163$}\textrm{.}  \end{array}\right.\end{displaymath}
Moreover, if $q$ is prime and $D(f)=-3,-8$, then equation (\ref{eq:maineqn}) has no integer solutions $x,y$ with $y>1$ for $q>3$.
\end{thm}
\begin{rem}
Monic quadratic polynomials $f,g$ have equal discriminant if and only if $f(x)=g(x+k)$ for some $k\in\mathbb{Z}$. It follows that in order to determine the integer solutions to the equations $y^q=g(x)$ where $D(g)$ is fixed, it is enough to determine them for one polynomial with discriminant $D(g)$ and then there is a bijection between the set of integer solutions of that equation and the set of integer solutions to another such equation.\end{rem}
An immediate corollary to Theorem \ref{thm:mainthm} is the following.
\begin{cor}\label{cor:intsolns}
Let $\vert y \vert >1$ be an integer which satisfies equation (\ref{eq:maineqn}) with $q>1$, then
\begin{enumerate}
\item[(a)] If $D(f)=-7$, the only solutions are
$$(y,q)\in\{(2,13),(2,5),(2,3),(\pm 2,2)\}.$$
\item[(b)] If $D(f)=-11$, the only solutions are $(y,q)\in\{(3,5),(\pm 3,2)\}$.
\item[(c)] If $D(f)=-19$, the only solutions are $(y,q)\in\{(5,7),(\pm 5,2)\}$.
\item[(d)] If $D(f)=-8$, the only solution is $(y,q)=(3,3)$.
\item[(e)] If $D(f)=-43$, the only solution is $(y,q)=(\pm 11,2)$.
\item[(f)] If $D(f)=-67$, the only solution is $(y,q)=(\pm 17,2)$.
\item[(g)] If $D(f)=-163$, the only solution is $(y,q)=(\pm 41,2)$.
\item[(h)] If $D(f)=-3$, the only solution is $(y,q)=(7,3)$. \end{enumerate}
\end{cor}
\begin{rem}
The procedure which is implemented to derive this corollary also works for the case $D(f)=-4$. There is no need for us to state it here; by our previous remark it suffices to study the equation \begin{displaymath}
y^q=x^2+1 \textrm{,} \end{displaymath}
which was shown to have no non-trivial solutions by Lebesgue in \cite{lebesgue1850}.\end{rem}
Bugeaud \cite{MR1838399} (with a correction by Bilu \cite{MR1930991}) proved that for $D_1,D_2$ squarefree positive integers, the only solutions of the Diophantine equation
\[D_1x^2+D_2^m=4y^n\] in positive integers $x,y,m$ odd, $n\geqslant 5$ prime with $\gcd(D_1x,D_2y)=1$ and $\gcd(n,h(\mathbb{Q}(\sqrt{-D_1D_2})))=1$ are given by \[
(y,n)\in\{(2,5),(2,7),(2,13),(3,5),(3,7),(4,7),(5,7)\}\textrm{.}\]
Our main results can be extracted from Bugeaud's. The approach here also uses the deep result of Bilu, Hanrot and Voutier \cite{MR1863855} but is more explicit in that we identify (in the case $D_1=m=1$) the equations where each of these powers appear. Many special forms of this equation have been considered via similar methods in the papers \cite{MR1916082,MR1928911} and the survey article \cite{MR2286850}.
\subsection{Applications}
The above results can be used in the explicit determination of all perfect power terms in sequences generated by certain quadratic polynomials. The study of perfect power terms in integer sequences is becoming increasingly popular. In \cite{MR697495}, it is shown using Baker-type estimates that any non-degenerate binary linear recurrence sequence has only a finite number of terms which are perfect powers, and in \cite{MR915504} the same result was proven for non-degenerate $n$-th order linear recurrences. In \cite{MR0163867} it is shown that the only squares in the Fibonacci sequence are 0,1,144 and in \cite{MR733078} it is shown using transcendence methods that the only cubes in the Fibonacci sequence are 0,1,8. In \cite{MR2215137} it is shown that 0,1,8,144 are the only perfect power terms in the Fibonacci sequence, which was a long standing open problem. This result uses a combination of Baker theory and the modular method which has grown out of Wiles' proof of Fermat's last theorem. In addition, for results on perfect powers in arithmetic progressions see the papers \cite{MR2205718,Gyory,MR2134477} which also use a combination of classical methods and the modular method. \newline\newline The sequences that interest us are the following.
\begin{defi} Let $g_m(x)=x^2-mx+m$ where $m\in\mathbb{N}$ and furthermore choose $a\in\mathbb{N}$ such that $a>m$ and $\gcd(a,m)=1$. Fix $m\neq 0,4$ and define a sequence $G^{(m)}(a)=(G_n^{(m)}(a))_{n\geqslant 0}$ where $G_n^{(m)}(a)=g_m^n(a)$, where $g_m^n$ denotes the $n$-th iterate of $g_m$. We will call $G^{(m)}(a)$ a generalised Sylvester sequence of type $m$. \end{defi}
\begin{rem}
Note that the assumption $a>m$ will ensure that the terms of these sequences are positive and strictly increasing, and therefore non-periodic. \end{rem}
This class of sequences contains, as special cases, the Fermat numbers \cite[A000058]{Sloanesite} which is $G^{(2)}(3)$ and the Sylvester sequence $G^{(1)}(2)$ \cite[A000215]{Sloanesite}. It was shown by Mohanty in \cite{MR514325} that the $n$-th term $G_n^{(m)}(a)$, of a generalised Sylvester sequence of type $m$ satisfies the following special recurrence relation,
\begin{displaymath}
G_n^{(m)}(a)=m+(a-m)G_0^{(m)}(a)G_1^{(m)}(a)...G_{n-1}^{(m)}(a)\textrm{,}\end{displaymath}
which when combined with an easy congruence condition allows one to show that any two distinct terms in this sequence are coprime. Aside from this the Sylvester sequence has some unusual properties which make it especially interesting. The Sylvester sequence has the property that its $n$-th term is the closest integer to $H^{2^n}$ for some real number $H>0$, see \cite{Sloanesite}. In fact this property holds for all sequences $G^{(1)}(a)$, and in general $G_n^{(m)}(a)$ is the closest integer to $H_1^{2^n}+\frac{m-1}{2}$ for some real number $H_1>0$, (see \cite{Flatters}) which can be derived from the work done in \cite{MR0148605}. The Sylvester sequence gives a way of obtaining infinitely many Egyptian fraction representations of 1, see \cite{Sloanesite}. A consequence of Corollary \ref{cor:intsolns} is that the Sylvester sequence has no terms which are perfect powers. This fact seems not to have been previously established, all that has been known is that there are no terms in $G^{(1)}(2)$ which are squares, \cite{Sloanesite}. In fact we can deduce much more.
\begin{cor} \label{cor:gstp1}
The only perfect power terms in a generalised Sylvester sequence of type 1 are $G_0^{(1)}(a)$ when $a$ is itself a perfect power, and $G_1^{(1)}(19)$. \end{cor}
We therefore know exactly which inputs give rise to perfect power terms and the position of these perfect powers in the sequence. The methods which we apply can also be used to give results for generalised Sylvester sequences of types 2 and 3.
\section{Proof of Main Results}
 Throughout this section for $\alpha$ a quadratic algebraic integer we denote by $\bar{\alpha}$, the algebraic conjugate of $\alpha$ not equal to $\alpha$ and by $u_n(\alpha,\bar{\alpha})$, the expression $\frac{\alpha^n-\bar{\alpha}^n}{\alpha-\bar{\alpha}}$. In addition by $\langle x \rangle$ we will mean the principal ideal generated by $x$. We begin with the following definition which will allow us to state the theorem of Bilu, Hanrot and Voutier that is instrumental in the proof of Theorem \ref{thm:mainthm} and hence Corollary \ref{cor:intsolns}.
\begin{defi}
Let $A=(a_i)_{i\geqslant 1}$ be an integer sequence. We say that a prime $p$ is a primitive prime divisor of $a_n$ if $p\mid a_n$ but $p \nmid a_m$ for any $m<n$ with $a_m\neq 0$. \end{defi}
\begin{thm}[\bf Bilu, Hanrot and Voutier~ \cite{MR1863855}]\label{thm:BHVthm}
Let $a,b,n\in\mathbb{Z}$ with $4<n\leqslant 30$ and $n\neq 6$. Then, up to equivalence, all Lucas pairs $(\alpha,\bar{\alpha})=(\frac{a+\sqrt{b}}{2},\frac{a-\sqrt{b}}{2})$ and $n$ such that $u_n(\alpha,\bar{\alpha})$ fails to have a primitive prime divisor are listed in the following table.
\begin{center}
\begin{tabular}[c]{|c|l|}
\hline
$n$ & $(a,b)$  \\
\hline
5 & $(1,-7),(1,-11),(12,-76),(12,-1364)$ \\
\hline
7 & $(1,-19)$  \\
\hline
8 & $(2,-24),(1,-7)$  \\
\hline
10 & $(2,-8),(5,-3),(5,-47)$ \\
\hline
12& $(1,5),(1,-7),(1,-11),(2,-56),(1,-15),(1,-19)$  \\
\hline
13& $(1,-7)$  \\
\hline
18 & $(1,-7)$ \\
\hline
30 & $(1,-7)$  \\
\hline
\end{tabular}
\end{center}
In particular, for all Lucas pairs $(\alpha,\bar{\alpha})$, $u_n(\alpha,\bar{\alpha})$ has a primitive prime divisor for each $n>30$.
\end{thm}
\begin{rem}
Here two Lucas pairs $(\alpha,\bar{\alpha})$ and $(\beta,\bar{\beta})$ are said to be equivalent if $\frac{\alpha}{\beta}=\frac{\bar{\alpha}}{\bar{\beta}}=\pm 1$. Thus, it is clear that if $(\alpha,\bar{\alpha})$ and $(\beta,\bar{\beta})$ are equivalent, then $u_n(\alpha,\bar{\alpha})=u_n(\beta,\bar{\beta})$ for all $n\in\mathbb{N}$.
\end{rem}
Now we have everything needed to prove the results stated in the introduction.
\begin{proof}[Proof of Theorem \ref{thm:mainthm1}]
Factorise $f(x)$ over the ring of integers $R$ of $\mathbb{Q}(\sqrt{D(f)})$ to obtain
\begin{displaymath}
y^q=x^2+ax+b=(x-\alpha)(x-\bar{\alpha}) \textrm{.} \end{displaymath}
where $\alpha=\frac{-a+\sqrt{D(f)}}{2}$. The ring $R$ is not a unique factorisation domain, so we work with ideals. Hence
\begin{equation}\label{eq:idealeqn}
\langle y \rangle^q=\langle x-\alpha \rangle \langle x-\bar{\alpha}\rangle\textrm{.}
\end{equation}
Now let $\mathfrak{p}$ be a prime ideal dividing both $\langle x-\alpha \rangle,\langle x-\bar{\alpha} \rangle$, so
\begin{displaymath}
x-\alpha\in\mathfrak{p} \quad \textrm{ and} \quad x-\bar{\alpha}\in\mathfrak{p} \end{displaymath}
and thus
\[ \alpha-\bar{\alpha}=\sqrt{D(f)}\in \mathfrak{p}\textrm{.} \]
It follows that $D(f)\in\mathfrak{p}$. In addition, $\mathfrak{p}\mid(x-\alpha)(x-\bar{\alpha})=y^q$, thus $\mathfrak{p}\mid y$ and hence $y\in\mathfrak{p}$. We now claim that $D(f)$ and $y$ are coprime. Suppose not. Then $\gcd(D(f),y)=p_1\cdots p_r$, for some $r\in\mathbb{N}$ where the $p_i$'s are distinct primes. Also
\[ 4y^q=(2x+a)^2-D(f) \] and so $p_1\cdots p_r\mid (2x+a)$. Then for some integers $k_1,k_2,k_3$ with $p_i\nmid k_3$ for each $i$, we have
\[ 4p_1^q\cdots p_r^qk_1=p_1^2\cdots p_r^2k_2+p_1\cdots p_rk_3 \] so that
\[4p_1^{q-1}\cdots p_r^{q-1}k_1=p_1\cdots p_rk_2+k_3\]
which implies $p_i\mid k_3$, a contradiction. Hence $D(f),y$ are coprime and thus there exist integers $m,n\in\mathbb{Z}$ with
\[ my+nD(f)=1 \textrm{.}\]
This in turn implies that $1\in\mathfrak{p}$, a contradiction to the fact that $\mathfrak{p}$ is prime. Hence $\langle x-\alpha \rangle, \langle x-\bar{\alpha}\rangle$ are coprime, and so by (\ref{eq:idealeqn}), we have
\[ \langle x-\alpha \rangle =I^q \]
for some integral ideal $I$. Let $q$ be coprime to $h$, then as $I^q$ is principal, $I$ too must be principal. Therefore
\[ \langle x-\alpha \rangle=\langle \beta \rangle ^q =\langle \beta^q \rangle \]
for some $\beta\in R$. From which we easily deduce
\[
x-\alpha=\epsilon\beta^q \] for some unit $\epsilon$ of $R$. In $R$, the only units are $\pm 1$ and so are themselves $q$-th powers. Thus
\[ x-\alpha=\delta^q \textrm{,} \] for some $\delta\in R$.
By applying the non-trivial Galois automorphism one obtains
\[ x-\bar{\alpha} =\bar{\delta}^q \]
and thus
\[ \bar{\alpha}-\alpha= \delta^q-\bar{\delta}^q\textrm{.} \]
which in turn yields
\begin{displaymath}
-\sqrt{D(f)}=k\sqrt{D(f)}u_q(\delta,\bar{\delta})\textrm{,} \end{displaymath}
where $k$ is some rational integer. Now $u_q(\delta,\bar{\delta})$ is an integer, hence $k=\pm 1$ which gives
\[
u_q(\delta,\bar{\delta})=\pm 1 \textrm{.} \]
Therefore the $q$-th term of the Lucas sequence $(u_n(\delta,\bar{\delta}))_{n\geqslant 1}$ fails to have a primitive divisor, so by Theorem \ref{thm:BHVthm}, coupled with the facts $D(f)$ is squarefree, $h>1$ and $q$ is prime means that $q\leqslant 3$. We have not taken into account the case where $q$ is not coprime to $h$. When $q\mid h$, the above method does not apply as $I^q$ principal does not imply $I$ principal in general. So for these values of $q$, we need to solve the equation $y^q=f(x)$ by hand.
\end{proof}
\begin{rem}
Actually, the proof of the above theorem tells us that we only need to check the cases $q=2,3$ and $q$ is a prime divisor of $h$. So we only need check at most $2+\omega(h)$ cases, where $\omega(h)$ is the number of prime factors of $h$. \end{rem}
\begin{proof}[Proof of Theorem \ref{thm:mainthm}] \label{thm:SEDE}
Let $f(z)=z^2+az+b$ be such that $a,b\in\mathbb{Z}$ and with discriminant $D(f)$ where $-D(f)\in\{7,11,19,43,67,163\}$. Let $R$ be the ring of integers of the splitting field of $f$ over $\mathbb{Q}$. In addition, let us assume $q>30$ and that there are integers $x,y$ which satisfy equation (\ref{eq:maineqn}). We have the factorisation of $f(x)$ as $(x-\alpha)(x-\bar{\alpha})$, for $\alpha=\frac{-a+\sqrt{D(f)}}{2}$. We may assume that $x-\alpha$ and $x-\bar{\alpha}$ are coprime in $R$. If $x-\alpha$ and $x-\bar{\alpha}$ have a common factor $d\in R$, then $d$ has to divide $\sqrt{D(f)}$, which is a prime of $R$ as $D(f)$ is a rational prime. This means that $d$ is either a unit or a unit multiple of $\sqrt{D(f)}$. Assume that $d=\pm\sqrt{D(f)}$ then we can re-write our equation $y^q=f(x)$ as \begin{equation} \label{eq:noncprime}
y^q=D(f)\left (\frac{x-\alpha}{\sqrt{D(f)}} \right ) \left( \frac{x-\bar{\alpha}}{\sqrt{D(f)}} \right )\textrm{.} \end{equation}
Now the terms on the RHS of (\ref{eq:noncprime}) are pairwise coprime. We know that the two bracketed terms are coprime so all we need to check is that $D(f)$ has no factors in common with $A=\frac{x-\alpha}{\sqrt{D(f)}}$ say. Suppose that $\epsilon$ is a non-trivial common factor then as $\epsilon\mid D(f)$ we have $\epsilon=\pm\sqrt{D(f)}$ or $\pm D(f)$. Applying the non-trivial Galois automorphism tells us that $\bar{\epsilon}=\pm \epsilon$. Now $\bar{\epsilon}\mid \bar{A}$, so $\epsilon\mid \bar{A}$ also. This is a contradiction and so $\epsilon$ must be a unit. Hence, as $R$ is a unique factorisation domain, each of $D(f),A,\bar{A}$ is a unit multiple of a $q$-th power. However, since $q>2$, $D(f)$ is not a $q$-th power so there are no integer solutions to (\ref{eq:noncprime}). Now assume that $x-\alpha$ and $x-\bar{\alpha}$ are coprime in $R$. Then equation (\ref{eq:maineqn}) implies that
\begin{displaymath}
x-\alpha=\pm \beta ^q \textrm{,} \end{displaymath}
for some $\beta\in R$. Once again $\pm 1$ are $q$-th powers so
\begin{displaymath}
x-\alpha=\gamma^q \end{displaymath}
for some $\gamma\in R$. Applying the non-trivial Galois automorphism yields
\begin{displaymath}
x-\bar{\alpha}=\bar{\gamma}^q\textrm{.} \end{displaymath}
The last two equations imply that
\begin{displaymath}
\bar{\alpha}-\alpha=\gamma^q-\bar{\gamma}^q\textrm{,} \end{displaymath}
which in turn yields
\begin{displaymath}
-\sqrt{D(f)}=k\sqrt{D(f)}u_q(\gamma,\bar{\gamma})\textrm{,} \end{displaymath}
for some $k\in\mathbb{Z}$. Once again $k=\pm 1$ which gives
\begin{equation}\label{eq:finaleq}
u_q(\gamma,\bar{\gamma})=\pm 1 \textrm{.} \end{equation}
As before we can now apply the result of Theorem \ref{thm:BHVthm} to conclude that $u_q(\gamma,\bar{\gamma})$ has a primitive prime divisor for all $q>30$ and equation (\ref{eq:finaleq}) is therefore untenable. Hence for (\ref{eq:maineqn}) to be soluble in integers $x,y$ we require that $q\leqslant 30$.
\newline\newline We will now prove the statement for $D(f)=-7$; the other cases follow similarly. Assume that we have a solution to (\ref{eq:maineqn}) for $q>13$. Then we know that equation (\ref{eq:finaleq}) holds for some $\gamma\in\mathbb{Z}\left[\frac{1+\sqrt{-7}}{2}\right]$, and the $q$th term in the Lucas sequence $(u_n(\gamma,\bar{\gamma}))_{n\geqslant 1}$ fails to have a primitive prime divisor. From Theorem \ref{thm:BHVthm} we have a complete list of conjugate pairs $(\gamma,\bar{\gamma})$ and positive integers $n$ such that $u_n(\gamma,\bar{\gamma})$ fails to have a primitive prime divisor. By the equivalence condition of Theorem \ref{thm:BHVthm} we may assume that the only candidate for $\gamma$ is $\frac{1+\sqrt{-7}}{2}$, consequently the only $n>13$ for which $u_n(\gamma,\bar{\gamma})$ fails to admit a primitive prime divisor, are $n=18,30$. Computing the values of $u_n\left (\frac{1+\sqrt{-7}}{2},\frac{1-\sqrt{-7}}{2}\right )$ for $n=18,30$, the values $\pm 1$ are never obtained, so there are no solutions to (\ref{eq:finaleq}) when $q>13$, which establishes the result for $D(f)=-7$.\newline\newline We are now left to prove the claims when $D(f)=-3,-8$. Let $f(x)=x^2+ax+b$ be an integral polynomial of discriminant $-3$ (as before, the case $D(f)=-8$ is similar). Then for integers $x,y$, $y>1$ with
\begin{displaymath}
y^q=x^2+ax+b\textrm{,}
 \end{displaymath} we have the factorisation
\begin{equation} \label{eq:DE1}
y^q=(x-\alpha)(x-\bar{\alpha}) \textrm{,}
\end{equation}
where $\alpha=\frac{-a+\sqrt{-3}}{2}$. Note that the RHS of equation (\ref{eq:DE1}) lies in the ring $\mathbb{Z}[\omega]$. As before we may assume that $x-\alpha$ and $x-\bar{\alpha}$ are coprime in $\mathbb{Z}[\omega]$. Therefore, we have from equation (\ref{eq:DE1})
\begin{equation}\label{eq:DE2}
x-\alpha=\delta \cdot \gamma^q \end{equation}
where $\gamma,\delta \in \mathbb{Z}[\omega]$ with $\delta$ a unit. \newline\newline
Since $\delta$ is a unit, it is itself a $q$-th power (since $q$ is a prime larger than 3) so we can absorb $\delta$ into the $\gamma$, and from equation (\ref{eq:DE2}) we have
\begin{equation} \label{eq:DE4}
x-\alpha = \epsilon^q \textrm{,} \end{equation}
 for some $\epsilon\in\mathbb{Z}[\omega]$. Again applying the non-trivial Galois automorphism gives
\begin{equation}\label{eq:DE5}
x-\bar{\alpha}=\bar{\epsilon}^q \textrm{.}\end{equation}
Subtracting (\ref{eq:DE5}) from (\ref{eq:DE4}) gives \begin{displaymath}
\epsilon ^q - \bar{\epsilon}^q=\bar{\alpha}-\alpha=-\sqrt{-3}\textrm{.} \end{displaymath}
Factorising the LHS of the above gives
\begin{displaymath}
(\epsilon-\bar{\epsilon})u_q(\epsilon,\bar{\epsilon})=-\sqrt{-3} \textrm{.} \end{displaymath}
Note that $\epsilon-\bar{\epsilon}=c\sqrt{-3}$ for some integer $c$. Hence,
\begin{displaymath}
cu_q(\epsilon,\bar{\epsilon})=-1{,} \end{displaymath}
and since $u_q(\epsilon,\bar{\epsilon})$ is an integer we know $c \mid 1 $. So $c=\pm 1$ and we end up with the following equation
\begin{equation} \label{eq:DE7}
u_q(\epsilon,\bar{\epsilon})=\pm 1\textrm{.} \end{equation}
It follows immediately from Theorem \ref{thm:BHVthm}, that $u_q(\epsilon,\bar{\epsilon})$ has a primitive prime divisor for all $q>30$, so for (\ref{eq:DE7}) to hold we must have $q<30$. Moreover, the only pairs $(n,\gamma)\in\mathbb{N}\times\mathbb{Z}[\omega]$ with $n>4$ such that $u_n(\gamma,\bar{\gamma})$ fails to have a primitive prime divisor are $\left (10,\pm \left ( \frac{5\pm\sqrt{-3}}{2}\right ) \right )$. Since $q$ is prime, we conclude that there are no solutions to (\ref{eq:DE7}).
\end{proof}
Now consider Corollary \ref{cor:intsolns}. We will prove only the case that $D(f)=-7$ since this is the situation which gives rise to the highest bound for the exponent $q$.
\begin{proof}[Proof of Corollary \ref{cor:intsolns}]
Note that for the equation \begin{displaymath}
y^q=f(x) \end{displaymath}
to have integer solutions $x,y$ with $\vert y \vert >1$ we have that equation (\ref{eq:finaleq}) holds, where $\beta\in\mathbb{Z}\left[\frac{1+\sqrt{-7}}{2} \right ]$. As in the proof of Theorem \ref{thm:mainthm} the only candidate for $\beta$ is $\frac{1+\sqrt{-7}}{2}$. When $\beta=\frac{1+\sqrt{-7}}{2}$, we see that equation (\ref{eq:finaleq}) for $q\geqslant 5$ is only satisfied for $q=5,13$. It may be assumed that $q$ is prime. Therefore, to fully solve this equation we need only look at the cases $q=2,3,5,13$. Without loss take $f(x)=x^2+x+2$, since it has discriminant $-7$. First solve the equation $y^2=f(x)$ in integers $x,y$. Completing the square gives
\begin{displaymath}
y^2=\left (x+\frac{1}{2} \right )^2+7/4 \end{displaymath}
and multiplying through by 4 gives
\begin{displaymath}
(2y)^2=(2x+1)^2+7 \end{displaymath} and so
\begin{displaymath}
(2y-2x-1)(2y+2x+1)=7\textrm{.} \end{displaymath}
So it is clear that $(2y-2x-1)=\pm 1,\pm 7$ and running through the possibilities yields that $x=-2,1$ and $y=2$.
\newline\newline
Note that the zeros of $f$ are $\alpha=\frac{-1+\sqrt{-7}}{2}$ and $\bar{\alpha}=\frac{-1-\sqrt{-7}}{2}$. So we have
\begin{displaymath}
y^q=(x-\alpha)(x-\bar{\alpha})\textrm{.}\end{displaymath}
As in the proof of Theorem {\ref{thm:mainthm}} assume that the two factors on the RHS of the equation are coprime. Therefore
\begin{displaymath}
x-\alpha =\pm \beta^q \textrm{,}\end{displaymath}
for some $\beta\in\mathbb{Z}[\frac{1+\sqrt{-7}}{2}]$. We only need to check the cases that $q=3,5,13$. Now $-1$ is a perfect $q$-th power, so
\begin{displaymath}
x+\frac{1-\sqrt{-7}}{2}=\epsilon^q \textrm{,} \end{displaymath}
for some $\epsilon\in\mathbb{Z}[\frac{1+\sqrt{-7}}{2}]$. Write $\epsilon=\frac{U+V\sqrt{-7}}{2}$ and by substituting in the above
\begin{equation} \label{eq:refeqn}
2^qx+2^{q-1}-2^{q-1}\sqrt{-7}=(U+V\sqrt{-7})^q \textrm{.} \end{equation}
First deal with the case $q=13$. Expanding out the bracket in equation (\ref{eq:refeqn}) and equating real and imaginary parts yields
\begin{eqnarray}
\label{eq:firstq13} f_1(U,V)=-4096 \\
\nonumber \\
\label{eq:secondq13} g_1(U,V)=8192x+4096 \end{eqnarray}
where $f_1(U,V)=V(13U^{12}-2002U^{10}V^2+63063U^8V^4-588588U^6V^6+1716715U^4V^8-1310946U^2V^{10}+117649V^{12})$ \newline\newline and \newline\newline $g_1(U,V)= U^{13}-546U^{11}V^2+35035U^9V^4-588588U^7V^6+3090087U^5V^8-4806802U^3V^{10}+1529437UV^{12}$.
\newline\newline From (\ref{eq:firstq13}), $V\mid 4096$. Using the polroots command in PARI, see \cite{pari}, we compute polroots$\left (f(x)+\frac{4096}{V}\right )$ where
\begin{displaymath}
f(x)=\frac{f_1(x,V)}{V} \end{displaymath} where $V$ is fixed and takes on the values $\pm 2^d$ where $d$ runs from 0 to 12 inclusive. Picking out the integer solutions yields
\begin{displaymath}
V=1,U=\pm 1\textrm{.}
\end{displaymath}
This implies that
\begin{displaymath}
g_1(U,V)=\pm 741376 \end{displaymath}
which, by substituting into (\ref{eq:secondq13}), yields $x=-91$ or 90 and so we conclude that $y=2$. By substituting $q=5$ into equation (\ref{eq:refeqn}) and solving in the same way we find that the only solutions to \begin{displaymath}
y^5=x^2+x+2 \end{displaymath}
are $x=-6,5$ and $y=2$. Similarly for $q=3$ we find that the only solutions to
\begin{displaymath}
y^3=x^2+x+2 \end{displaymath}
are $x=-3,2$ and $y=2$. Therefore, from the remark below Theorem \ref{thm:mainthm}, the only solutions to the equation $y^q=f(x)$ where $f(x)$ has discriminant $-7$ are $y=2$ and $q=2,3,5,13$ and so we have proven part (a) of Corollary \ref{cor:intsolns}.
\end{proof}

\section{Applications to polynomially generated sequences}
In this section we show how Theorems \ref{thm:mainthm1},\ref{thm:mainthm} and Corollary \ref{cor:intsolns} can be applied to deduce perfect power results for generalised Sylvester sequences of types 1,2 and 3.

\begin{proof}[Proof of Corollary~\ref{cor:gstp1}] We are looking to solve the equation $y^q=x^2-x+1$ in integers $x,y,q$ where $y,q>1$. Since $x^2-x+1$ has discriminant equal to $-3$ we see from Corollary \ref{cor:intsolns} that the only integer solution $(y,q)$ to this equation is $(7,3)$. Hence if we have a perfect power term in such a sequence, the previous term $x$ must satisfy
\begin{displaymath}
x^2-x+1=343\textrm{.} \end{displaymath}
Solving the previous equation gives $x=-18,19$. So to see 343 appearing in our generalised Sylvester sequence of type 1, we need the previous term needs to be 19 since all terms in the sequence are positive. However, 19 is not the image of any integer under the mapping $z\rightarrow z^2-z+1$, so if we have 343 appearing it must be because we have chosen 19 as our initial input. This concludes the proof.
\end{proof}
\begin{rem}
From the above result, we see that the Sylvester sequence has no perfect powers since it is $G^{(1)}(2)$.
\end{rem}
\begin{lem}
The only perfect power terms in a generalised Sylvester sequence of type 2 are $G_0^{(2)}(a)$ when $a$ is a perfect power. \end{lem}
\begin{proof}
We are looking for integer solutions to the equation
\begin{displaymath}
y^q=x^2-2x+2\textrm{.} \end{displaymath}
The polynomial on the RHS of the above has discriminant equal to $-4$, and so by the remark below the proof of Theorem \ref{thm:mainthm}, we can invoke Lebesgue's result, \cite{lebesgue1850} to show that this equation has no integer solution $(x,y)$ with $y>1$.
\end{proof}
\begin{cor}
The only perfect power terms in a generalised Sylvester sequence of type 3 are $G_0^{(3)}(a)$ when $a$ is a perfect power, and $G_1^{(3)}(20)$. \end{cor}
\begin{proof}
To see this, we simply observe that generalised sequences of type 3 are more or less the same as those of type 1.
Since $x^2-3x+3=(x-1)^2-(x-1)+1$, we see that $G^{(3)}(a)=G^{(1)}(a-1)$. The statement of this corollary then follows from Corollary \ref{cor:gstp1}.
\end{proof}
We finish with an example to illustrate that the bound in Theorem \ref{thm:mainthm1} is sharp, and to show how we can find all power terms in the sequences coming from this polynomial by iterating it upon an integral input.
\begin{exa}
Let $f(x)=x^2+x+6$. We wish to solve the equation $y^q=f(x)$ in integers $x,y$, $q\geqslant 2$. Note that $D(f)=-23$, so $f(x)$ satisfies the hypotheses of Theorem \ref{thm:mainthm1} and we conclude at once that $q\leqslant 3$ since $h(\mathbb{Q}(\sqrt{-23}))=3$. To show that our bound is sharp we need only show there are solutions when $q=3$. When $q=3$ the equation defines an elliptic curve and the equation can be solved by using the MAGMA package, \cite{MR1484478}.
We find that the integer solutions to the equation are
\[
\begin{aligned}
(x,y)\in&\{(22,8),(-23,8),(-42,12),(41,12),(-2,2),(1,2),(14,6),\\&(-15,6),(3625,236),
(-3626,236)\}\textrm{.}
\end{aligned}
\]
The case $q=2$ is straightforward, we can rearrange the equation a little to obtain
\begin{displaymath}
(2y-2x-1)(2y+2x+1)=23. \end{displaymath}
Using the fact that the two factors on the LHS are factors of 23 gives that the only solutions in this case are
\begin{displaymath}
(x,y)\in\{(5,\pm 6),(-6,\pm 6)\}\textrm{.} \end{displaymath}
As before we can now use this information to show that for $n\geqslant 1$, $a\in\mathbb{Z}$,
$f^{n}(a)$ is a perfect power exactly when $n=1$ and
$$
a=-3626,-42,-23,-15,-6,-2,1,5,14,22,41,3625.
$$
Hence no term beyond the first in the sequence $(f^{n}(a))_{n\geqslant 1}$ is a perfect power.
\end{exa}


\def\cprime{$'$}

\end{document}